\documentclass[10pt]{amsart}
\usepackage[pagebackref]{hyperref}

\renewcommand{\d}{\mathrm{d}}%exterior derivative
\renewcommand{\L}{\Lambda}
\renewcommand{\Im}{\operatorname{Im}}
\renewcommand{\Re}{\operatorname{Re}}
\newcommand{\ts}{\textstyle }

\newcommand{\cJ}{{\mathbb J}}
\newcommand{\bbR}{{\mathbb R}}
\newcommand{\bbC}{{\mathbb C}}
\newcommand{\bbP}{{\mathbb P}}

\newcommand{\E}{{\mathrm e}}
\newcommand{\iC}{{\mathrm i}}
\newcommand{\GL}{\operatorname{GL}}

\newcommand{\SO}{\operatorname{SO}}
\newcommand{\SU}{\operatorname{SU}}
\newcommand{\Un}{\operatorname{U}}
\newcommand{\Or}{\operatorname{O}}
\newcommand{\Gr}{\operatorname{Gr}}

\newcommand{\Sp}{\operatorname{Sp}}

\newcommand{\Diff}{\operatorname{Diff}}

\newcommand{\phm}{\phantom{-}}

\newcommand{\Js}{\mathsf{J}}

\newcommand{\bi}{\bar\imath}
\newcommand{\bj}{\bar\jmath}
\newcommand{\bk}{\bar k}
\newcommand{\bell}{\bar\ell}

\newcommand{\w}{{\mathchoice{\,{\scriptstyle\wedge}\,}{{\scriptstyle\wedge}}
      {{\scriptscriptstyle\wedge}}{{\scriptscriptstyle\wedge}}}}

\newcommand{\Gtwo}{\ifmmode{{\rm G}_2}\else{${\rm G}_2$}\fi}

\newcommand{\be}{\begin{equation}}
\newcommand{\ee}{\end{equation}}
\newcommand{\bpm}{\begin{pmatrix}}
\newcommand{\epm}{\end{pmatrix}}

\numberwithin{equation}{section}

\newtheorem{theorem}{Theorem}
\newtheorem{lemma}{Lemma}

\theoremstyle{remark}
\newtheorem{definition}{Definition}

\newtheorem{remark}{Remark}

\begin{document}

\author[R. Bryant]{Robert L. Bryant}
\address{Duke University Mathematics Department\\
         P.O. Box 90320\\
         Durham, NC 27708-0320}
\email{\href{mailto:bryant@math.duke.edu}{bryant@math.duke.edu}}
\urladdr{\href{http://www.math.duke.edu/~bryant}%
         {http://www.math.duke.edu/\lower3pt\hbox{\symbol{'176}}bryant}}

\title[Chern's six-sphere study]
      {S.-s. Chern's study of \\
       almost-complex structures\\
        on the six-sphere}

\date{April 5, 2021}

\begin{abstract}
In April 2003, S.-s.~Chern began a study of
almost-complex structures on the $6$-sphere, 
with the idea of exploiting the special properties 
of its well-known almost-complex structure invariant
under the exceptional group $\Gtwo$.  
While he did not solve the (currently still open) problem of
determining whether there exists an integrable almost-complex
structure on~$S^6$, he did prove a significant identity
that resolves the question for an interesting class of
almost-complex structures on~$S^6$.
\end{abstract}



\thanks{
Thanks to Duke University for its support via a research grant 
and to the National Science Foundation 
for its support via grants DMS-0103884 and~DMS-1359583.  
\hfill\break
\hspace*{\parindent} 
This is Draft Version~$2.0$, begun on February 6, 2005, 
and last edited on April 5, 2021.
}

\maketitle

\setcounter{tocdepth}{2}
\tableofcontents

\section{Introduction}\label{sec: intro}

In the spring of 2003, Professor Shiing-shen Chern turned
his attention to the old problem of determining whether or not
there is a complex structure on the $6$-sphere.%
\footnote{According to a result of Borel and Serre~\cite{BorelSerre1953},
the only spheres that support even an almost-complex structure 
are $S^2$ and $S^6$.  Of course, $S^2$ supports a complex structure,
unique up to diffeomorphism.  While $S^6$ famously supports 
an almost-complex structure, as will be discussed below, 
all attempts to date to determine whether it supports 
a complex structure have been unsuccessful.
N.B.: Throughout this article, 
in view of the Newlander-Nirenberg Theorem~\cite{NewNiren}, 
I will use the terms `complex structure'
and `integrable almost-complex structure' as synonyms.}

Chern's approach was quite original:  He chose to use the known,
very homogeneous, almost-complex structure~$\Js$ on~$S^6$ 
as a sort of `background reference' 
and to analyze the integrability equations 
for an arbitrary almost-complex structure~$J$ on $S^6$
by studying the invariants of $J$ relative to~$\Js$.  
As was his wont, he used the structure equations 
and the moving frame associated to $\Gtwo$, the symmetry group of $\Js$, 
to study the consequences of integrability of~$J$.  

As he carried out the process of frame adaptation 
and differentiation of the structure equations, 
he discovered a rather unexpected and remarkable identity
and found that he could use this to eliminate some possibilities
for complex structures on~$S^6$.  
He privately circulated a manuscript detailing his calculations
and arguments.  Upon examination of his argument, however, 
it was realized that Professor Chern had inadvertently made
a hidden assumption about the algebraic invariants of~$J$ 
relative to the reference~$\Js$, so that his identity did not hold
universally and hence could not, without some further new idea, 
be used to resolve the main existence question. 

Nevertheless, undiscouraged, Professor Chern continued his work 
on the main problem, and I had the privilege of corresponding
with him about his further ideas in this area 
up through the fall of~2004.  
Unfortunately, Professor Chern passed away on December 3, 2004,
before he could fully develop his further ideas.

While I knew that Professor Chern's approach had not solved
the problem, I found the identity he had discovered 
(via a clever and delicate moving frames calculation) 
to be quite interesting and felt that it deserved to be recorded 
and remembered.
Unfortunately, other duties at the time delayed my collecting 
and clarifying my thoughts on the matter.  Now, as the twentieth
anniversary of Professor Chern's passing approaches, 
I have been reminded of this issue and so was motivated 
to prepare this note about Professor Chern's idea 
and his main result in this area.

It is simple to state the result:  The symmetry group of~$\Js$
preserves both a metric~$g$ and a $2$-form~$\omega$ on~$S^6$.  
A.~Blanchard~\cite{Blanchard1953} had proved in~1953
that there is no complex structure on~$S^6$ 
that is compatible with the metric~$g$.%
\footnote{
Unfortunately, Blanchard's result was apparently forgotten. 
It was rediscovered independently by C.~LeBrun~\cite{LeBrun1987}
in~1987.} 
Chern's identity implies 
that there is no complex structure on~$S^6$ 
that is compatible with the $2$-form~$\omega$.  

It turns out that these two cases are quite different:
The condition of compatibility with~$g$ is a system of~$12$
pointwise algebraic equations on an almost-complex structure~$J$.
As Blanchard's analysis shows, the integrability conditions 
for such almost-complex structures are an involutive system 
whose general local solution depends on three holomorphic functions 
of three complex variables.  
In contrast, the condition of compatibility with~$\omega$ 
is a system of only~$6$ pointwise algebraic conditions 
on an almost-complex structure~$J$. 
Chern's identity shows that the integrability conditions 
for such almost-complex structures do \emph{not} form an involutive system; 
indeed, his computation uncovers the nonvanishing torsion 
that proves its non-involutivity.
  
In this note, I explain Chern's result (see Theorem~\ref{thm: chern}) 
and give the proof, basically along the lines he originally proposed, 
but with a few simplifications that became clear in hindsight.

Meanwhile, in~2013, N.~Daurtseva, 
while studying complex structures in dimension~$6$
that are compatible with the canonical $2$-form
of a nearly-K\"ahler structure (a more general situation
than that considered by Chern), independently re-discovered
Chern's results~\cite{Daurtseva2014}, though her analysis
proceeded along somewhat different lines.%
\footnote{I was unaware of her results in~2014
when I posted an earlier version of this note to the arXiv.
}
 
Once one realizes what Chern's calculation means, 
there is a way to get to the essential identity 
without having to carry out the frame adaptations and normalizations 
that made Chern's original argument somewhat difficult to follow.
Indeed, one then sees that Chern's argument applies verbatim
to prove a more general statement about compatible complex structures 
on what I have called 
\emph{elliptic definite almost-symplectic $6$-manifolds}, 
a class of structures that includes, for example, 
all strictly nearly-K\"ahler $6$-manifolds.
In the final section of the note, 
I make some remarks about these and related matters.

\section{The structure equations}\label{sec: streqs}

This section will collect the main results about the group \Gtwo\
that will be needed.  The reader may consult~\cite{Bryant1982},
for details concerning the properties of the 
group \Gtwo\ that are not proved here.

\subsection{The group~\Gtwo}  Let~$e_1,e_2,\ldots,e_7$ denote
the standard basis of~$\bbR^7$ (whose elements will be referred to as 
column vectors of height $7$) and let~$e^1,e^2,\ldots,e^7:\bbR^7\to\bbR$ 
denote the corresponding dual basis.

For notational simplicity, write~$e^{ijk}$ for the wedge 
product~$e^i\w e^j\w e^k$ in~$\L^3\bigl((\bbR^7)^*\bigr)$.  
Define
\begin{equation}\label{eq: phi def}
\phi = e^{123}+e^{145}+e^{167}+e^{246}-e^{257}-e^{347}-e^{356}.
\end{equation}
It is a theorem of W.~Reichel (see \cite{Agricola2008} for the history) 
that the subgroup of~$\GL(7,\bbR)$ that fixes~$\phi$ is a compact, 
connected, simple Lie group of type~\Gtwo. In this article, 
this result will be used to justify the following definition:

\begin{definition}[The group \Gtwo]\label{def: G2}
\begin{equation}\label{eq: G2 def}
\Gtwo = \left\{\ g\in \GL(7,\bbR)\ \vrule\ g^*(\phi) = \phi\ \right\}.
\end{equation}
\end{definition}

\subsection{Associated structures}\label{sssec: assoc strucs}
I will list here a few properties of \Gtwo\ 
that will be needed in this article. 

\Gtwo\ preserves the metric and orientation on~$\bbR^7$ 
for which the basis~$e_1,e_2,\ldots,e_7$ is an oriented orthonormal basis.  

\Gtwo\ acts transitively on the unit sphere $S^6\subset\bbR^7$,
and the \Gtwo-stabilizer of any~$u\in S^6$ 
is isomorphic to~$\SU(3)\subset\SO(6)$; 
thus, $S^6 = \Gtwo/\SU(3)$.  
Since $\SU(3)$ acts transitively on~$S^5\subset\bbR^6$, 
it follows that \Gtwo\ acts transitively 
on the set of orthonormal pairs of vectors in~$\bbR^7$. 

\Gtwo\ preserves the cross-product operation, which is defined
as the unique bilinear mapping~$\times:\bbR^7\times\bbR^7\to\bbR^7$ 
that satisfies 
$$
(u\times v) \cdot w = \phi(u,v,w).
$$
It follows that $u\times v = - v\times u$ is perpendicular
to both $u$ and $v$, and, since \Gtwo\
acts transitively on orthonormal pairs, the evident identities
$e_1\times e_2 = e_3$ and $e_1\times(e_1\times e_2) = - e_2$ 
imply that $u\times(u\times v) = (u\cdot v)\, u - (u\cdot u)\,v$
for all $u,v\in\bbR^7$.

\Gtwo\ acts simply transitively on the set of orthonormal triples~$(u,v,w)$
in $\bbR^7$ that satisfy $(u\times v)\cdot w = 0$.

\subsection{The standard almost-complex structure}\label{ssec: stdJ}
The above formulae imply that there is an almost-complex structure
$\Js:TS^6\to TS^6$ on~$S^6$ that is invariant under the action of
\Gtwo\ defined by the formula
$$
\Js_u(v) = u\times v
\qquad \text{for all}\ v\in T_uS^6 = u^\perp \subset \bbR^7
$$
In fact, it is not difficult to show (see~\cite{Bryant1982})
that the group of almost-complex automorphisms of $(S^6,\Js)$
is equal to \Gtwo\ .

\subsection{The moving frame}
Let~$g:\Gtwo\to\SO(7)$ be the inclusion mapping and write~$g = (g_i)$
where~$g_i:\Gtwo\to\bbR^7$ is the $i$-th column of~$g$.  Set
\be
x = g_1,\ f_1 = {\ts\frac12}(g_2-\iC\,g_3), 
\ f_2 = {\ts\frac12}(g_4-\iC\,g_5),
\ f_3 = {\ts\frac12}(g_6-\iC\,g_7),
\ee 
Then, as shown in~\cite{Bryant1982}, there are left-invariant, complex-valued 
$1$-forms~$\theta_i$ ($1\le i\le 3$) and~$\kappa_{i\bj}$ ($1\le i,j\le 3$) 
defined on \Gtwo\ and
satisfying~$\kappa_{i\bj} = -\overline{\kappa_{j\bi}}$ 
and~$\kappa_{1\bar1}+\kappa_{2\bar2}+\kappa_{3\bar3}= 0$ 
such that the \emph{first structure equations} hold, i.e.,
\be\label{eq: 1ststreqs}
\begin{aligned}
\d x &= -2\iC\,f_j\,\theta_j + 2\iC\,\bar f_j\,\overline{\theta_j}\,,\\
\d f_i &= -\iC\,x\,\overline{\theta_i} + f_\ell\,\kappa_{\ell\bi}
             -\epsilon_{ijk}\,\overline{f_j}\,\theta_k\,,
\end{aligned}
\ee
where $\epsilon_{ijk}$ is skew-symmetric in its indices and $\epsilon_{123}=1$. 
These complex-valued $1$-forms 
satisfy the \emph{second structure equations}
\be\label{eq: 2ndstreqs}
\begin{aligned}
\d \theta_i &= -\kappa_{i\bell}\w\theta_\ell 
                + \epsilon_{ijk}\,\overline{\theta_j}\w\overline{\theta_k},\\
\d\kappa_{i\bj} &= -\kappa_{i\bk}\w\kappa_{k\bj} 
                  + 3\,\theta_i\w\overline{\theta_j} 
                    -\delta_{ij}\,\theta_k\w\overline{\theta_k}\,.
\end{aligned}
\ee

It will be useful to have the structure equations in matrix form.  
Set
\be\label{eq: matrixforms}
\theta = \begin{pmatrix}\theta_1\\ \theta_2\\ \theta_3\\ \end{pmatrix}\,,
\ 
\Theta = \begin{pmatrix}\theta_2\w\theta_3\\
                        \theta_3\w\theta_1\\
                        \theta_1\w\theta_2\\ \end{pmatrix}\,,
\quad\text{and}\quad
\kappa = \begin{pmatrix}
         \kappa_{1\bar1}&\kappa_{1\bar2}&\kappa_{1\bar3}\\
         \kappa_{2\bar1}&\kappa_{2\bar2}&\kappa_{2\bar3}\\
         \kappa_{3\bar1}&\kappa_{3\bar2}&\kappa_{3\bar3}\\ 
          \end{pmatrix}\,.
\ee
Then the second structure equations take the compact form
\be\label{eq: 2ndstreqsmatrixform}
\begin{aligned}
\d \theta &= -\kappa\w\theta + 2\,\overline{\Theta},\\
\d\kappa &= -\kappa\w\kappa 
                  + 3\,\theta\w {}^t\overline{\theta} 
                    -{}^t\theta\w\overline{\theta}\,\mathrm{I}_3\,.
\end{aligned}
\ee

\subsection{\Gtwo-invariant forms}\label{ssec: G2invforms}
The structure equations~\eqref{eq: 2ndstreqs} 
imply that the almost-complex structure~$\Js$ on~$S^6$
has (and is defined by) the property
that a complex-valued $1$-form~$\alpha$ on~$S^6$ is of $\Js$-type~$(1,0)$
if and only if $x^*\alpha$ is a linear combination of the~$\theta_i$.

The standard metric~$g$ induced on~$S^6$ by its inclusion
into~$\bbR^7$ has the form
\be\label{eq: g_as_theta}
x^*g = 4\,{}^t\theta\circ\bar\theta.
\ee
The $2$-form~$\omega$ on $S^6$ that is associated to~$\Js$ via~$g$ 
satisfies
\be
x^*\omega = 2i\,{}^t\theta\w\bar\theta\,.
\ee
This $2$-form is not closed; by the structure equations,
\be
\d\omega = 3 \Im\bigl(\Upsilon\bigr),
\ee
where~$\Upsilon$ is complex-valued $3$-form on~$S^6$ that satisfies
\be
x^*\Upsilon = 8\,\theta_1\w\theta_2\w\theta_3\,.
\ee
Note that $\Upsilon$ is of $\Js$-type~$(3,0)$.
It is not closed, but satisfies
\be
\d\Upsilon = 2\,\omega^2.
\ee

The forms~$\omega$, $\Re\Upsilon$, and $\Im\Upsilon$ 
generate the ring of \Gtwo-invariant forms on~$S^6$.  
In fact, as explained in~\cite{Bryant2005}, 
the form $\omega$ determines~$\Upsilon$, $\Js$,
and the metric~$g$, and, consequently, 
the (pseudo-)group of (local) diffeomorphisms of~$S^6$ 
that preserve~$\omega$ is the (pseudo-)group 
generated by the action of~\Gtwo. 

\section{Some linear algebra}

\subsection{Compatibility}\label{ssec: linalgcompat}
On a given vector space~$V$ of even dimension over~$\bbR$, 
there is a well-known set of relations among the set of
quadratic forms on~$V$, the set of symplectic structures on~$V$,
and the set of complex structures on~$V$.

A pair $(g,J)$ consisting of a nondegenerate inner product~$g$ on~$V$
and a complex structure~$J:V\to V$ is said to be \emph{compatible}
if $g(Jv,Jw) = g(v,w)$, or equivalently, $g(Jv,w)+g(v,Jw)=0$.
In particular, the bilinear form~$\omega$ on~$V$
defined by~$\omega(v,w) = g(Jv,w)$ is nondegenerate and satisfies
$\omega(w,v) = -\omega(v,w)$, so $\omega$ is a symplectic structure on~$V$.  

A pair~$(\omega,J)$ consisting of a symplectic form~$\omega$ on~$V$
and a complex structure~$J:V\to V$ is said to be \emph{compatible} 
if $\omega(Jv,Jw)=\omega(v,w)$, 
or, equivalently, $\omega(v,Jw)=\omega(w,Jv)$.
In particular, the bilinear form~$g$ on$V$ 
defined by~$g(v,w)=\omega(v,Jw)$ is nondegenerate and satisfies
$g(w,v)=g(v,w)$, so $g$ is a nondegenerate inner product on~$V$.
One says that an $\omega$-compatible complex structure~$J$ 
has $\omega$-index~$(p,q)$ (where $p{+}q=n$) if the $J$-compatible metric 
defined by~$g(v,w) = \omega(v,Jw)$ has inertial index~$(2p,2q)$.  
Note that, if an $\omega$-compatible $J$ has $\omega$-index~$(p,q)$, 
then $-J$ (which is automatically $\omega$-compatible) 
has $\omega$-index~$(q,p)$.
 
For a vector space~$V$ over~$\bbR$ of dimension~$2n$,
the space of complex structures on~$V$ is a homogeneous space
$\GL(V)/\GL(V,J)\simeq \GL(2n,\bbR)/\GL(n,\bbC)$ 
and, hence, is a smooth manifold of dimension~$2n^2$.

For a symplectic structure~$\omega$ 
on a real vector space~$V$ of dimension~$2n$, 
the set of $\omega$-compatible complex structures on~$V$ 
has $n{+}1$ connected components:  
The space of $\omega$-compatible complex structures on~$V$ 
with $\omega$-index~$(p,q)$ 
is a connected homogeneous space~$\Sp(V,\omega)/\Un(V,\omega,J)
\simeq \Sp(n,\bbR)/\Un(p,q)$, and hence has dimension~$n^2{+}n$.  

Correspondingly, for a nondegenerate inner product~$g$ 
on a real vector space~$V$ of dimension~$2n$, 
the set of $g$-compatible complex structures on~$V$ 
is empty unless $g$ has inertial index~$(2p,2q)$ 
for some $(p,q)$ with $p{+}q=n$.
When $g$ does have index~$(2p,2q)$, 
the set of $g$-compatible complex structures on~$V$ 
is a homogeneous space~$\Or(V,g)/\Un(V,g,J)
\simeq \Or(2p,2q)/\Un(p,q)$ and, hence, has dimension~$n^2{-}n$.

In particular, note that, for a nondegenerate inner product~$g$
of even index~$(2p,2q)$ on a vector space~$V$ of real dimension~$2n$,
the space of $g$-compatible complex structures on~$V$ 
is a submanifold of codimension $n^2{+}n$ 
in the space of all complex structures on~$V$, 
while, for a symplectic structure~$\omega$ on~$V$, 
the space of $\omega$-compatible complex structures on~$V$ 
is a submanifold of codimension $n^2{-}n$
in the space of all complex structures on~$V$.

\subsection{The almost-complex structures on a manifold}
The linear algebra considerations in~\S\ref{ssec: linalgcompat}
have natural analogs for structures on smooth manifolds.

The set of (smooth) almost-complex structures on
a smooth manifold~$M$ of dimension~$2n$
is the (possibly empty) set of (smooth) sections 
of a canonically defined smooth bundle~$\pi:\cJ(M)\to M$ 
whose fiber over a point~$x\in M$ is the space 
of complex structures on the real vector space~$T_xM\simeq \bbR^{2n}$ 
and, hence, is identifiable 
with the homogeneous space~$\GL(2n,\bbR)/\GL(n,\bbC)$,
which has dimension~$2n^2$.  
When~$\pi:\cJ(M)\to M$ does have continuous sections, 
one says that the set of almost-complex structures on~$M$ 
depends on~$2n^2$ functions of~$2n$ variables.

When~$M$ is endowed with a Riemannian metric~$g$, 
one can consider the subbundle~$\cJ(M,g)\subset\cJ(M)$ 
whose sections are the $g$-compatible almost-complex structures on~$M$.
The fiber of the subbundle~$\pi:\cJ(M,g)\to M$ over a point~$x\in M$
is identifiable with $\Or(2n)/\Un(n)$, 
and, hence, the subbundle~$\cJ(M,g)$ has codimension~$n^2{+}n$ in~$\cJ(M)$.
It is easy to show that, if the bundle~$\pi:\cJ(M)\to M$ does have sections,
then so does~$\pi:\cJ(M,g)\to M$, so that the set of $g$-compatible
almost-complex structures on~$M$ 
depends on $n^2{-}n$ functions of $2n$ variables.

Meanwhile, when $M$ is endowed with a nondegenerate $2$-form~$\omega$
(not assumed to be closed), 
one can consider the subbundle~$\cJ(M,\omega)\subset\cJ(M)$ 
whose sections are the $\omega$-compatible almost-complex structures on~$M$.
This subbundle is the disjoint union of subbundles~$\cJ_q(M,\omega)$,
where the fiber of $\cJ_q(M,\omega)$ over~$x\in M$ 
consists of the $\omega_x$-compatible complex structures on~$T_xM$ 
that have $\omega_x$-index~$(n{-}q,q)$, 
and hence is identifiable with~$\Sp(n,\bbR)/\Un(n{-}q,q)$. 
Hence, the subbundle~$\cJ_q(M,\omega)$ has codimension~$n^2{-}n$ in~$\cJ(M)$.
It is easy to show that, if the bundle~$\pi:\cJ(M)\to M$ does have sections,
then so does~$\pi:\cJ_0(M,\omega)\to M$, 
so that the set of $\omega$-compatible almost-complex structures on~$M$ 
depends on $n^2{+}n$ functions of $2n$ variables.

\subsection{Special features when $M=S^6$}\label{ssec: JonS6}
The fibers of the bundle~$\pi:\cJ(S^6)\to S^6$ have dimension~$18$.

$\cJ(S^6,g)\subset\cJ(S^6)$ is a subbundle of codimension~$12$
that is a bundle deformation retract of $\cJ(S^6)$, 
so that any almost-complex structure~$J$ on~$S^6$ is homotopic to one 
that is orthogonal with respect to the standard metric~$g$ on~$S^6$.

Using the (unique) spin structure on~$S^6$,  
one can show (see~\cite{Bryant2005}) 
that any $g$-orthogonal almost-complex structure~$J$ on~$S^6$ is homotopic
through $g$-orthogonal almost-complex structures to either~$\Js$ or~$-\Js$
(depending on which of the two orientations~$J$ induces on~$S^6$).

Consequently, there are exactly two homotopy equivalence classes
of almost-complex structures on~$S^6$, $[\Js]$ and $[-\Js]$.

Meanwhile,~$\cJ(S^6,\omega)\subset\cJ(S^6)$ is a subbundle of codimension~$6$
and is the disjoint union of the subbundles~$\cJ_q(S^6,\omega)$
for $0\le q\le 3$.  $\Js$ is a section of~$\cJ_0(S^6,\omega)$
while $-\Js$ is a section of~$\cJ_3(S^6,\omega)$.  The following
simple lemma will be used below.

\begin{lemma}\label{lem: no_split_J}
The subbundles~$\cJ_1(S^6,\omega)$ and~$\cJ_2(S^6,\omega)$ 
have no continuous sections over~$S^6$.
\end{lemma}

\begin{proof}
A continuous section~$J$ of~$\cJ_1(S^6,\omega)$ 
would allow one to construct a nondegenerate quadratic form 
of inertial index~$(4,2)$ on~$S^6$, 
and this would imply that the tangent bundle of~$S^6$ 
can be written as a direct sum~$TS^6 = E_4\oplus E_2$ 
of two oriented subbundles, $E_4$ of rank~$4$ and $E_2$ of rank~$2$.  
However, because the Euler class of~$TS^6$ is nonzero, 
such a splitting would imply $0\not= e(TS^6) = e(E_4)e(E_2) = 0$.

If $J$ were a continuous section of~$\cJ_2(S^6,\omega)$, then
$-J$ would be a continuous section of~$\cJ_1(S^6,\omega)$.
\end{proof}

Finally, note that the two subbundles $\cJ(S^6,g)$ and $\cJ(S^6,\omega)$
intersect transversely 
along the images of the sections~$\pm\Js:S^6\to\cJ(S^6)$, in fact,
$$
\cJ(S^6,g)\cap \cJ_0(S^6,\omega) = \phm\Js(S^6)\subset \cJ(S^6),
$$
while
$$
\cJ(S^6,g)\cap \cJ_3(S^6,\omega) = -\Js(S^6)\subset \cJ(S^6).
$$
Meanwhile, the subbundles $\cJ(S^6,g)$ and $\cJ_i(S^6,\omega)$
for $i=1,2$ do not intersect transversely.  
In fact, for $i=1$~or~$2$, the intersections
$$
\cJ_i(S^6,\omega,g) = \cJ_i(S^6,\omega)\cap \cJ(S^6,g)
$$  
are bundles over~$S^6$ whose fibers are diffeomorphic 
to~$\SU(3)/\Un(2)\simeq \bbC\bbP^2$ and so have dimension~$4$.
For further discussion of these bundles, see Remark~\ref{rem: propersub}.

\section{\Gtwo-invariants of almost-complex structures}\label{sec: invs}

Consider an arbitrary smooth almost-complex structure~$J$ on~$S^6$.
While there are no zeroth-order invariants of~$J$ 
under the action of the group~$\Diff(S^6)$, 
there are zeroth-order invariants 
under the action of the isometry group~$\SO(7)$
and even more under the action of the group~$\Gtwo$.  
This corresponds to the fact (discussed above)
that the space of all complex structures on~$\bbR^6$, 
which is a homogeneous space~$\GL(6,\bbR)/\GL(3,\bbC)$
of dimension~$18$, is not homogeneous under the action
of~$\SO(6)$, much less that of~$\SU(3)$.

Let $\pi: F_J\to S^6$ be the right principal $\GL(3,\bbC)$-bundle over~$S^6$
whose elements are the $J$-linear isomorphisms~$u:T_{\pi(u)}S^6\to\bbC^3$.
The action of $\GL(3,\bbC)$ on~$F_J$ is given by $u\cdot A = A^{-1}\circ u$.
Let $\eta$ be the canonical $\bbC^3$-valued $1$-form on~$F_J$; it is 
defined by the formula
$$
\eta(v) = u\bigl(\pi'(v))
$$
for all~$v\in T_uF_J$.

Now let~$B_J\subset \Gtwo\times F_J$ be the pullback bundle over $S^6$
consisting of the pairs~$(g,u)$ such that $x(g) = \pi(u)$.  
It is a principal right $\SU(3){\times}\GL(3,\bbC)$-bundle over~$S^6$,
and its projection onto either $\Gtwo$ or $F_J$ 
is a surjective submersion with connected fibers.  
From now on, all forms defined on either $F_J$, $\Gtwo$, or~$S^6$ 
will be regarded as pulled back to $B_J$ without notating the pullback. 

There exist unique mappings~$r,s:B_J\to M_{3,3}(\bbC)$
such that
\be
\begin{aligned}
\theta &= r\,\eta + s\,\overline{\eta}\\
\bar\theta &= \overline{s}\,\eta + \overline{r}\,\overline{\eta}.
\end{aligned}
\ee
Because of the linear independence of the~$\eta^i$ and the~$\overline{\eta^i}$ as well as the $\theta_j$ and $\overline{\theta_j}$,
\be
\det\bpm r & s\\ \bar s & \bar r \epm \not=0.
\ee
(Note that this determinant is real valued;
in fact, $J$ induces the same orientation on~$S^6$ as $\Js$ 
if and only if this determinant is positive.)

Now, the canonical forms satisfy
\be\label{eq: theta-eta-action}
R^*_{(g,h)}\theta = g^{-1}\theta
\qquad\text{and}\qquad R^*_{(g,h)}\eta = h^{-1}\eta
\ee 
for $g\in \SU(3)$ and $h\in\GL(3,\bbC)$, which implies that
\be\label{eq: r-s-action}
R^*_{(g,h)}r = g^{-1}r h
\qquad\text{and}\qquad R^*_{(g,h)}s = g^{-1}s\overline{h}.
\ee

\subsection{Symmetric tensorial invariants}\label{ssec: symmtensorinvs}
Chern~\cite{Chern2004} observed that the formulae
\eqref{eq: theta-eta-action} and \eqref{eq: r-s-action} imply 
that there exist two positive semi-definite $J$-hermitian 
quadratic forms~$P_J$ and~$Q_J$ on~$S^6$ that satisfy
\be
\begin{aligned}
P_J &= {}^t(r\,\eta)\circ \overline{(r\,\eta)} 
= {}^t\eta\,\bigl({}^tr\overline{r}\bigr)\,\overline{\eta}\\
Q_J &= {}^t(\bar{s}\,\eta)\circ(s\,\overline{\eta})
= {}^t\eta\,\bigl({}^t\overline{s} s\bigr)\,\overline{\eta}.
\end{aligned}
\ee
While one does not know \emph{a priori} whether~$P_J$ or~$Q_J$ 
might be positive definite, the sum~$P_J+Q_J 
={}^t\eta\,\bigl({}^tr\overline{r}+{}^t\overline{s}s\bigr)
\,\overline{\eta}$ is positive definite because 
the invertible matrix
$$
\bpm {}^tr & {}^t\bar s\\ {}^ts & {}^t\bar r \epm 
\bpm \bar r & \bar s\\ s & r \epm
= \bpm {}^tr\overline{r}+{}^t\overline{s}s & {}^t r\bar s + {}^t\bar s r\\ 
 {}^t s\bar r + {}^t\bar r s & {}^ts\bar s+ {}^t\bar r r \epm
$$
is Hermitian positive definite.  In addition, 
there exists a $J$-complex quadratic form~$\gamma_J$ on~$S^6$ 
that satisfies
\be 
\gamma_J = {}^t(r\,\eta)\circ(\bar{s}\,\eta)
= {\ts\frac12}\,{}^t\eta\,\bigl({}^tr\bar{s}+{}^t\bar{s} r\bigr)\,\eta.
\ee

Note that the metric~$g=4\,{}^t\theta\circ\bar\theta$ on~$S^6$ satisfies
$g= 4\,\gamma_J + 4\,(P_J{+}Q_J) + 4\,\overline{\gamma_J}$,
implying that $\gamma_J$ and $P_J{+}Q_J$ are tensors on~$S^6$ 
that depend only on $J$ and the metric~$g$.  In fact, the above
formulae show that the $J$-type decomposition of~$g$ is given by
\be\label{eq: gJinvariants}
\begin{aligned}
g^{(2,0)}_J &= 4\,\gamma_J 
 = 2\,{}^t\eta\,\bigl({}^tr\bar{s}+{}^t\bar{s} r\bigr)\,\eta\\
g^{(1,1)}_J &= 4\,(P_J+Q_J)
 = 4\,{}^t\eta\,\bigl({}^tr\bar{r}+{}^t\bar{s}s\bigr)\,\bar{\eta}\\
g^{(0,2)}_J &= 4\,\overline{\gamma_J}
 = 2\,{}^t\bar\eta\,\bigl({}^t\bar{r}s+{}^ts\bar{r}\bigr)\,\bar\eta\,.
\end{aligned}
\ee

\subsection{Alternating tensorial invariants}\label{ssec: alttensorinvs}
Meanwhile, for the \Gtwo-invariant differential forms, 
one has a $J$-type decomposition of~$\omega$ as
\be\label{eq: omegaJtypes}
\omega = \omega^{(2,0)}_J + \omega^{(1,1)}_J + \omega^{(0,2)}_J
\ee 
where, using the above notation,
\be
\begin{aligned}
\omega^{(2,0)}_J 
&=\phantom{2}i\,\,{}^t\eta\w\bigl({}^tr\bar{s}-{}^t\bar{s} r\bigr)\,\eta\\
\omega^{(1,1)}_J
&=2i\,\,{}^t\eta\w\bigl({}^tr\bar{r}-{}^t\bar{s}s\bigr)\,\bar{\eta}\\
\omega^{(0,2)}_J 
&=\phantom{2}i\,\,{}^t\bar\eta\w\bigl({}^t\bar{r}s-{}^ts\bar{r}\bigr)
\,\bar\eta\,.
\end{aligned}
\ee
Also, the $3$-form~$\Upsilon$ has a $J$-type
decomposition as a sum of terms~$\Upsilon^{(p,q)}_J$.  
In what follows, only the following two formulae will be needed: 
\be\label{eq: UpsilonJtypes}
\begin{aligned}
\Upsilon^{(3,0)}_J &= 8 \det(r)\,\eta_1\w\eta_2\w\eta_3\,,\\
\Upsilon^{(0,3)}_J &= 8 \det(s)\,
\overline{\eta_1}\w\overline{\eta_2}\w\overline{\eta_3}\,.
\end{aligned}
\ee

\section{Integrability}\label{sec: integrab}

The major unsolved question concerning almost-complex structures on~$S^6$
is whether there exists an \emph{integrable} one. 

By the celebrated theorem of Newlander and Nirenberg~\cite{NewNiren},
in order for an almost-complex structure~$J$ on a manifold~$M$ 
to be integrable, i.e., to be the almost-complex structure 
induced by an actual complex structure on~$M$,
it is necessary and sufficient that the ideal generated by the 
$J$-linear complex $1$-forms on~$M$ be closed under exterior differentiation.

This condition is equivalent to the vanishing of the Nijnhuis tensor of~$J$,
which is simply the condition that, when one writes the $J$-type
decomposition of the exterior derivative%
\footnote{By definition, $d^{p,q}_J\alpha = \pi^{k+p,l+q}_J(\d\alpha)$
when $\alpha$ has $J$-type $(k,l)$.}
$$
\d = \d^{2,-1}_J + d^{1,0}_J + d^{0,1}_J + d^{-1,2}_J\,
$$
the operator $d^{2,-1}_J$ vanishes identically 
(which implies that $d^{-1,2,}_J$ vanishes identically as well).

For example, since the structure equations imply
$\d^{2,-1}_{\Js}\omega = -{\textstyle\frac32}i\,\Upsilon\not=0$, 
it follows that the \Gtwo-invariant almost-complex structure~$\Js$ 
on~$S^6$ is not integrable; indeed, its Nijnhuis tensor is nowhere vanishing. 

\subsection{Metric compatibility}\label{ssec: metriccomp}
It follows from~\eqref{eq: gJinvariants} 
that $g$ is $J$-Hermitian, i.e., $J$ is orthogonal with respect to~$g$
(which is the same as saying that $(g,J)$ are compatible) 
if and only if $\gamma_J=0$.  As noted above, 
this condition constitutes~$12$ real equations on~$J$ 
that are of order~$0$.

\begin{theorem}[Blanchard~\cite{Blanchard1953}, LeBrun~\cite{LeBrun1987}]
\label{thm: blanchardlebrun}
There is no $g$-compatible complex structure on~$S^6$.
\end{theorem}  

\begin{proof} (Sketch.)
The arguments of Blanchard and LeBrun (which are similar) 
depend on the following idea 
(which works in all dimensions):  If $J$ is an almost-complex structure
on an open subset~$U$ in the $2n$-sphere~$S^{2n}\subset\mathbb{R}^{2n+1}$ 
that is orthogonal with respect to the standard metric on~$S^{2n}$, 
then one can define an embedding~$\tau:U\to\Gr_{n}(\mathbb{C}^{2n+1})$
by letting $\tau(u)\subset \mathbb{C}^{2n+1}$ be the complex $n$-plane
consisting of the vectors of the form~$v-i\,J_uv$ 
for~$v\in T_uS^n = u^\perp\subset\mathbb{R}^{2n{+}1}$.

LeBrun showed that, 
when $J$ is orthogonal with respect to the induced metric, 
the integrability of~$J$
is equivalent to the condition that $\tau:U\to\Gr_{n}(\mathbb{C}^{2n+1})$
have complex linear differential (when $J$ is used to define the
almost-complex structure on~$U$).  In particular, in the integrable case,
$\tau(U)$ is a complex submanifold of~$\Gr_{n}(\mathbb{C}^{2n+1})$ 
and hence inherits a K\"ahler structure from the standard K\"ahler structure
on~$\Gr_{n}(\mathbb{C}^{2n+1}) 
= \Un(2n{+}1)/\bigl(\Un(n){\times}\Un(n{+}1)\bigr)$.

Since $S^{2n}$ does not have a K\"ahler structure for~$n>1$ 
(because $H^2(S^{2n})=0$ when~$n>1$), 
it follows that, for~$n>1$, 
there cannot be an integrable complex structure on~$S^{2n}$ 
that is orthogonal with respect to the standard metric.%
\footnote{By~\cite{BorelSerre1953}, 
$S^{2n}$ has a continuous almost-complex structure 
only when $n=1$ or $n=3$, but this is a subtle, topological result.
An alternative proof of this result, which uses Bott Periodicity, 
can be found in \cite{LawsonMichelsohn}.}
\end{proof}

\subsection{$\omega$-compatibility}\label{ssec: chern}
As discussed in~\S\ref{ssec: JonS6}, 
the condition that $J$ be $\omega$-compatible 
constitutes~$6$ equations of order~$0$ on~$J$.  
Thus, $\omega$-compatibility is 
a much less restrictive condition than $g$-compatibility.  
Nevertheless, Chern's arguments~\cite{Chern2004}
show that there is no $\omega$-compatible complex structure on~$S^6$.

\begin{theorem}[Chern]\label{thm: chern}
Any $\omega$-compatible complex
structure on a connected open set~$U\subset S^6$ 
is a section of either $\cJ_1(U,\omega)$ or $\cJ_2(U,\omega)$.
In particular, there is no $\omega$-compatible complex structure on~$S^6$.
\end{theorem}

\begin{proof}  Let $J:U\to\cJ(U,\omega)$
be an integrable section over a connected open~$U\subset S^6$.

Then $\omega^{(2,0)}_J=0$, 
so that $\omega^{(0,2)}_J = \overline{\omega^{(2,0)}_J} =0$,
implying that $\omega = \omega^{(1,1)}_J$.  
As noted above, the hypothesis that~$J$ be integrable, then implies
$$
0 = \d^{2,-1}_J\omega = \pi^{3,0}_J(\d\omega) 
= \pi^{3,0}_J\bigl(3\Im(\Upsilon)\bigr) 
= \pi^{3,0}_J\bigl(-{\textstyle\frac32}i(\Upsilon-\overline{\Upsilon})\bigr).
$$
Thus, by~\eqref{eq: UpsilonJtypes},
\be\label{eq: 30_J_domega}
0 
%= \pi^{3,0}_J\bigl(-{\textstyle\frac32}i(\Upsilon{-}\overline{\Upsilon})\bigr)
  = {\textstyle\frac32}i\left(\,\overline{\pi^{0,3}_J(\Upsilon)} 
        {-}\pi^{3,0}_J(\Upsilon)\right)
  = 12i\,\bigl(\det(\bar s){-}\det(r)\bigr)\, \eta_1\w\eta_2\w\eta_3\,,
\ee
which yields the essential identity (first observed by S.-s.~Chern)
\be\label{eq: chernidentity}
\det(\bar s) = \det(r).
\ee

Since~$\omega = \omega^{(1,1)}_J 
= 2i\,\,{}^t\eta\w\bigl({}^tr\bar{r}-{}^t\bar{s}s\bigr)\,\bar{\eta}$
and since~$\omega$ is a nondegenerate $2$-form, it follows that
the Hermitian matrix~${}^tr\bar{r}-{}^t\bar{s}s$ has nonvanishing determinant.
Now, it cannot be either positive definite or negative definite:
If it were positive definite, then ${}^tr\bar{r}\ge {}^tr\bar{r}-{}^t\bar{s}s$ 
would be positive definite and hence $\det(\bar s)=\det(r)$ 
would be nonzero, so that ${}^t\bar{s}s$ would also be positive definite.  
By~\eqref{eq: chernidentity}, 
the two positive definite Hermitian matrices~${}^tr\bar{r}$ and~${}^t\bar{s}s$ would then have the same determinant (volume);
consequently, ${}^tr\bar{r}-{}^t\bar{s}s$ could not be positive definite.
A similar argument using ${}^tr\bar{r}-{}^t\bar{s}s\ge -{}^t\bar{s}s$ 
shows that ${}^tr\bar{r}-{}^t\bar{s}s$ cannot be negative definite either. 

Thus, $\omega^{(1,1)}_J$ must have have $J$-index either~$(2,1)$ or~$(1,2)$, 
and, by the connectedness of~$U$, this index must be constant.
Thus, $J$ must be a section 
of either~$\cJ_1(U,\omega)$ or~$\cJ_2(U,\omega)$, as claimed.

Finally, by Lemma~\ref{lem: no_split_J}, 
there are no global sections of~$\cJ_1(S^6,\omega)$ or~$\cJ_2(S^6,\omega)$,
so there is no $\omega$-compatible complex structure on~$S^6$.
\end{proof}

\begin{remark}[A proper subbundle]\label{rem: propersub}
As the reader will have no doubt realized, 
Chern's basic observation amounts to the fact 
that any integrable~$J$ that satisfies~$\omega^{(0,2)}_J = 0$
must also satisfy~\eqref{eq: 30_J_domega} and that this 
constitutes two additional \emph{algebraic} equations on~$J$.
As the above proof shows, this condition 
implies not only that such a $J$ must be a section 
of~$\cJ_1(S^6,\omega)\cup\cJ_2(S^6,\omega)$, 
but that its image must lie in a certain proper subset
$\cJ'_1(S^6,\omega)\cup \cJ'_2(S^6,\omega)$. 

Now, it is not difficult 
to show that each subset~$\cJ'_i(S^6,\omega)\subset\cJ_i(S^6,\omega)$ 
for $i=1,2$ is a nonempty, smooth subbundle 
of codimension~$2$ in~$\cJ_i(S^6,\omega)$
(and hence is of codimension~$8$ in~$\cJ(S^6)$). 

Of course, it is not clear at this point whether or not there might
be even further restrictions on local complex structures 
that are sections of~$\cJ'_i(S^6,\omega)$
that are derivable by careful use of the structure equations.  
It seems that much of Chern's further computation in the different
versions of his manuscript and communications with me about this
problem were attempts to derive such further integrability conditions 
by use of the structure equations, but, as far as I can tell, these
were inconclusive.  The generality of the set of local integrable
sections of~$\cJ'_i(S^6,\omega)$ remains unknown as of this writing.

However, one can see that the set of local integrable sections
of~$\cJ'_i(S^6,\omega)$ is nonempty as follows:
A part of the structure equations of~$\Gtwo$ can be written as
\be\label{eq: G2quadric}
\begin{aligned}
\d\theta_1 & = -\kappa_{1\bar1}\w\theta_1 
               -\kappa_{1\bar2}\w\theta_2 
               -\kappa_{1\bar3}\w\theta_3 
               + 2\,\overline{\theta_2}\w\overline{\theta_3}\\
\d\overline{\theta_2} &= 
                \kappa_{2\bar2}\w\overline{\theta_2}
               +\kappa_{1\bar3}\w\overline{\theta_3}
               +\kappa_{1\bar2}\w\overline{\theta_1}
               + 2\,\overline{\theta_2}\w\theta_1\\
\d\overline{\theta_3} &= 
                \kappa_{3\bar3}\w\overline{\theta_3}
               +\kappa_{2\bar3}\w\overline{\theta_2}
               +\kappa_{1\bar3}\w\overline{\theta_1}
               + 2\,\overline{\theta_3}\w\theta_1\\
\d\kappa_{1\bar2} &= (\kappa_{2\bar2}-\kappa_{1\bar1})\w\kappa_{1\bar2}
                     +\kappa_{3\bar2}\w\kappa_{1\bar3}
                    +3\,\theta_1\w\overline{\theta_2}\\
\d\kappa_{1\bar3} &= (\kappa_{3\bar3}-\kappa_{1\bar1})\w\kappa_{1\bar3}
                     +\kappa_{2\bar3}\w\kappa_{1\bar2}
                    +3\,\theta_1\w\overline{\theta_2}
\end{aligned}
\ee
These equations imply that the system
$\theta_1 = \overline{\theta_2} = \overline{\theta_3} = \kappa_{1\bar2}
= \kappa_{1\bar3} = 0$ is Frobenius, so that its leaves
are the right cosets of a compact subgroup~$H\subset\Gtwo$ of dimension~$4$
that is isomorphic to~$\Un(2)$, 
and that there is a complex structure on~$\Gtwo/H$ 
for which the $(1,0)$-forms pull back to~$\Gtwo$ 
to be linear combinations of the complex-valued $1$-forms
in~$\{\theta_1,\overline{\theta_2},\overline{\theta_3},
\kappa_{1\bar2},\kappa_{1\bar3}\}$.  
In fact, as shown in~\cite{Bryant1982},
the resulting compact complex manifold~$Q=\Gtwo/H$ 
is biholomorphic to the complex $5$-quadric.

Moreover, there is an embedding of $\Gtwo/H$ into~$\cJ'_2(S^6,\omega)$
that takes the coset~$gH$ to the complex structure~$J(gH)$ on~$T_{x(g)}S^6$
whose $(1,0)$-forms pull back to $T^*_g\Gtwo$ via~$x:\Gtwo\to S^6 
= \Gtwo/\SU(3)$ to be linear combinations
of~$\bigl\{\theta_1,\overline{\theta_2},\overline{\theta_3}\bigr\}$.

One can verify without difficulty that this embeds $\Gtwo/H$ 
as a smooth subbundle of~$\cJ'_2(S^6,\omega)$ 
whose fiber over a point of~$S^6$ is diffeomorphic to~$\mathbb{CP}^2$.
In fact, the image is precisely the subbundle~$\cJ_2(S^6,\omega,g)$
as it was defined at the end of~\ref{ssec: JonS6}.
 
Thus, one has a $\Gtwo$-equivariant fibration~$Q\to S^6$.
While the differential of this projection is not complex linear 
with respect to the almost-complex structures on the domain and range, 
it does have the following property:  
A (local) section of this bundle 
whose image is a holomorphic submanifold of~$Q$ 
(and there are many such, essentially,
they depend on two holomorphic functions of three complex variables) 
represents an integrable almost-complex structure on the domain of the
section.  

Consequently, the subbundle~$Q\subset\cJ'_2(S^6,\omega)$
has many local sections that represent 
$\omega$-compatible complex structures on open sets~$U\subset S^6$.
Thus, whatever integrability conditions remain to be discovered 
for integrable sections of $\cJ'_i(S^6,\omega)$, 
they do not preclude the existence of local solutions.
\end{remark}

\begin{remark}[Almost-symplectic $6$-manifolds]
Chern's argument for~$(S^6,\omega)$ applies to a much wider class
of almost-symplectic $6$-manifolds, namely, to what I will call 
\emph{elliptic definite} almost-symplectic $6$-manifolds.

Recall that an \emph{almost-symplectic $6$-manifold}
is a $6$-manifold~$M$ endowed with a nondegenerate $2$-form~$\omega$.
Given such a structure, one has a unique decomposition
\be\label{eq: om_prim_decomp}
\d\omega = \lambda\w\omega + \pi
\ee 
where $\lambda$ is a $1$-form on~$M$ and $\pi$ is an $\omega$-primitive
$3$-form on~$M$, i.e., a $3$-form on~$M$ that satisfies $\omega\w\pi = 0$.

Of course, if $\pi$ vanishes identically, 
then taking the exterior derivative 
of~\eqref{eq: om_prim_decomp} yields $0 =\d\lambda\w\omega$, 
which implies that $\d\lambda=0$, so that, at least locally,
$\lambda = -\d f$. This implies that $\d(\E^{f}\omega)=0$, 
i.e., $\omega$ is `locally conformally symplectic'.

More interesting is the case when $\pi$ is `generic':  
It turns out (see the Appendix of~\cite{Bryant2005}, for example) 
that there are two open orbits of $\GL(V)$ acting on $\Lambda^3(V^*)$ 
when $V$ has dimension~$6$:
One, the `split type', consists of the $3$-forms 
equivalent to $e^1\w e^2 \w e^3 + e^4\w e^5\w e^6$ 
for some basis~$e^i$ of~$V^*$,
the other, the `elliptic type', consists of $3$-forms 
equivalent to $e^1\w e^3\w e^6{+}e^1\w e^4\w e^5
{+}e^2\w e^3\w e^5{-}e^2\w e^4\w e^6$,
i.e., to~$\Im\bigl((e^1{+}ie^2)\w(e^3{+}ie^4)\w(e^5{+}ie^6)\bigr)$ 
for some basis~$e^i$ of~$V^*$.   

One says that an almost-symplectic $6$-manifold~$(M,\omega)$ 
is of \emph{elliptic type} if~$\pi$, the $\omega$-primitive component 
of~$\d\omega$, is of elliptic type at each point of~$M$.

For example, $S^6$ 
endowed with its $\Gtwo$-invariant $2$-form~$\omega$ is of elliptic type, 
as the structure equations in~\S\ref{ssec: G2invforms} show. 
To be of elliptic type is an open condition on~$\omega$
in the $C^1$-topology on $2$-forms on~$M$.

Suppose now that $(M^6,\omega)$ is a connected almost-symplectic $6$-manifold
of elliptic type, and let~\eqref{eq: om_prim_decomp}, 
where $\omega\w\pi=0$, be the $\omega$-type decomposition of~$\d\omega$.
Then $\pi = 3\Im(\Upsilon)$ for some complex-valued, decomposable $3$-form
$\Upsilon$ on~$M$, 
and this $\Upsilon$ is unique up to replacement by~$-\bar\Upsilon$.
There is a unique almost-complex structure~$J$ on~$M$ 
for which $\Upsilon$ will have $J$-type~$(3,0)$; if one
replaces $\Upsilon$ by $-\bar\Upsilon$, then $J$ must be replaced by~$-J$.
Thus, the choice of $\Upsilon$ (and $J$) will be made unique by requiring 
that $J$ induce the same orientation on~$M$ as does the volume form~$\omega^3$.
Assume that this choice has been made.

Since~$\omega\w(\Upsilon-\bar\Upsilon) \equiv 0$, 
it follows that~$\omega$ must be of $J$-type~$(1,1)$ 
and hence represent an Hermitian form on~$(M,J)$.  
Since~$\omega^3>0$ induces the same orientation on~$M$ as~$J$, 
it follows that the signature of~$\omega$ as a $J$-Hermitian form
is either $(3,0)$ or $(1,2)$.
One says that $(M,\omega)$ is \emph{elliptic definite} 
if this signature is~$(3,0)$.  

It is immediate that the condition for a nondegenerate $2$-form~$\omega$
on~$M^6$ to define a elliptic definite almost-symplectic structure on~$M$
is open in the $C^1$-topology on $2$-forms on~$M$.  
Thus, if one such structure exists on~$M$, then there are many.

For example, $S^6$ endowed with its $\Gtwo$-invariant $2$-form~$\omega$
is elliptic definite.  (The other signature can also occur: 
For example, if one considers~$\Gtwo'\subset\SO(3,4)$, 
the non-compact real form of $\Gtwo$, it turns out that this group 
contains a subgroup isomorphic to~$\SU(1,2)$, and the homogeneous
space~$H^6 = \Gtwo'/\SU(1,2)$ carries a $\Gtwo'$-invariant $2$-form
that is of elliptic type and that has signature $(1,2)$ 
with respect to its natural almost-complex structure~$J$, 
as constructed above.)  

Any strictly nearly-K\"ahler $6$-manifold
defines an elliptic definite almost-symplectic structure.
In fact, Daurtseva~\cite{Daurtseva2014} 
gives a generalization of Chern's argument 
that applies to strictly nearly-K\"ahler $6$-manifolds.

For an elliptic definite almost-symplectic manifold~$(M^6,\omega)$,
the argument of Chern applies essentially without change, 
yielding that any $\omega$-compatible complex structure on~$M$ 
must be a section of either~$\cJ'_1(M,\omega)$ or~$\cJ'_2(M,\omega)$,
the codimension~$2$ subbundles of~$\cJ_1(M,\omega)$ or~$\cJ_2(M,\omega)$
that are defined by the analog of
Chern's determinantal relation~\eqref{eq: chernidentity}.
(Unlike the case of $S^6$, on other $6$-manifolds these bundles can have global
integrable sections.  This happens, for example, for the well-known
homogeneous strictly nearly-K\"ahler structures 
on $S^3{\times}S^3$, $\bbC\bbP^3$, and $F_{1,2}(\bbC^3) = \SU(3)/T^2$.)

For example, this argument implies that there is no
complex structure on $S^6$ that is compatible 
with the $2$-form associated to the strictly nearly-K\"ahler
structure of cohomogeneity~$1$ on $S^6$ constructed recently by 
Foscolo and Haskins~\cite{FoscoloHaskins2017}.

\end{remark}

\bibliographystyle{hamsplain}

\providecommand{\bysame}{\leavevmode\hbox to3em{\hrulefill}\thinspace}

\end{document}